\documentclass{amsart}

\usepackage{geometry}

\usepackage{amsmath, amsfonts, amssymb}
\usepackage{xspace}
\usepackage{graphicx,color}
\usepackage[curve]{xypic}
\usepackage{pifont}
\usepackage[parfill]{parskip}
\newtheorem{question}{Question} 
\newtheorem{theo}{Theorem}[section]

\newtheorem{prop}{Proposition}[section]

\newtheorem{lemme}[prop]{Lemma}

\newtheorem{rema}{Remark}
\newtheorem{exa}{Example}

\def\eqdef{=}

\def\Pf{{P_f}}
\def\Pg{{P_g}}

\def\C{{\mathbb C}}

\def\0{{\mathbb 0}}
\def\P{{{\mathbb P}}}

\def\R{{\mathbb R}}

\def\S{{\Sigma}}

\def \epsilon{\varepsilon}

\def\cal{\mathcal}

\def\teich{\mathrm{Teich}(\P^1,A)}

\def\mod{\mathrm{Mod}(\P^1,A)}
\def\tmod{\mathrm{Mod}(\P^1,{\Pf})}
\def\pteich{\mathrm{Teich}(\P^1,{\Pf})}
\def\pmod{\mathrm{Mod}(\P^1,{\Pf})}
\def\myteich{\mathrm{Teich}(\P^1,{\Pf})}
\def\Teich{\mathrm{Teich}}

\def\basepoint{\circledast}
\def\belyi{{s}}
\def\perm{{\mu}}
\def\perminv{{\nu}}

\def\id{{\rm id}}
\def\qed{{\hfill{$\square$}}}

\def\pmcg{\mathrm{PMCG}(\P^1, {\Pf})}

\def\pullsbackto{{\underset{f}{\leftarrow}}}
\newcommand{\bdry}{\partial}

\begin{document}

%*******************************************************************************
% Front Matter%
%   The table of contents follows which is generated by \tableofcontents.
%
%   For final production, the generated *.toc file must be renamed to toc.tex.
%   This file can be edited as needed like any other included tex file.  Make
%   sure that the \tableofcontents command remains commented-out for the final
%   version of the TOC.
%
%   Next is the preface file.  Edit the text in preface.tex.
%
%   Additional front matter sections can be added as needed such as lists of
%   symbols.
%*******************************************************************************
%\frontmatter
% \pagestyle{empty}
 %\include{front} % title and copyright pages

 %\pagestyle{fancy}
% \renewcommand{\chaptermark}[1]{\markboth{#1}{#1}}
% \renewcommand{\sectionmark}[1]{}
% \setcounter{page}{560}
% \tableofcontents % uncomment this line if you need to generate a TOC
% \include{toc} % use and edit for final version
 %\include{preface}

%*******************************************************************************
% Parts and chapters
%   \part{Part 1} where "Part 1" is the part title
%   \include{chapter1) where "chapter1" is the filename chapter1.tex
%*******************************************************************************
%\mainmatter
% \renewcommand{\chaptermark}[1]{\markboth{\thechapter. #1}{}}
% \renewcommand{\sectionmark}[1]{\markright{\thesection. #1}}
%---BEGIN AUTHOR EDIT---

%\include{BEKP_corrected-2}
%!TEX root = akp-BEKP.tex

\title{On Thurston's pullback map}
\author{Xavier Buff, Adam Epstein, Sarah Koch,
and Kevin Pilgrim}
\today 
\maketitle

\begin{abstract} Let $f: \P^1 \to \P^1$ be a rational map with finite
postcritical set $\Pf$.  Thurston showed that $f$ induces a
holomorphic map $\sigma_f: \pteich \to \pteich$ of the Teichm\"uller
space to itself.  The map $\sigma_f$ fixes the basepoint
corresponding to the identity map $\mathrm{id}: (\P^1, \Pf) \to
(\P^1, \Pf)$.  We give examples of such maps $f$ showing that
the following cases may occur:
\begin{enumerate}
\item the basepoint is an attracting fixed point, the image of $\sigma_f$ is open and dense
in $\pteich$ and the pullback map $\sigma_f:\pteich\to \sigma_f\bigl(\pteich\bigr)$
is a covering map,
\item  the basepoint is a superattracting fixed point, the image of $\sigma_f$ is $\pteich$
and $\sigma_f:\pteich\to \pteich$ is a ramified Galois
covering map, or
\item the map $\sigma_f$ is constant.
\end{enumerate}
\end{abstract}

\section{Introduction}\label{intro}

In this article, $\S$ is an oriented $2$-sphere. All maps $\S\to \S$
are assumed to be orientation-preserving. The map $f:\S\to \S$ is a
branched covering of degree $d\geq 2$. A particular case of interest
is when $\S$ can be equipped with an invariant complex structure for
$f$. In that case, $f:\Sigma\to \Sigma$ is conjugate to a rational
map $F:\P^1\to \P^1$.

According to the Riemann-Hurwitz formula, the map $f$ has $2d-2$
critical points, counting multiplicities. We denote $\Omega_f$ the
set of critical points and $V_f\eqdef f(\Omega_f)$ the set of
critical values of $f$. The {\em postcritical set} of $f$ is the set
\[ \Pf \eqdef \bigcup_{n>0}f^{\circ n}(\Omega_f).\]
The map $f$ is {\em postcritically finite} if $\Pf$ is finite.
Following the literature, we refer to such maps simply as {\em
Thurston maps}.

Two Thurston maps  $f:\S\to \S$ and $g:\S\to \S$ are {\em
equivalent} if there are homeomorphisms $h_0: (\S,\Pf) \to
(\S,{\Pg})$ and $h_1: (\S,\Pf) \to (\S,{\Pg})$ for which $h_0\circ
f=g\circ h_1$ and $h_0$ is isotopic to $h_1$ through homeomorphisms
agreeing on $\Pf$. In particular, we have the following commutative
diagram:
\[\xymatrix{
&(\S,\Pf)\ar[d]_{f}\ar[r]^{h_1} &
(\S,\Pg) \ar[d]^{g} \\
&(\S,\Pf)\ar[r]^{h_0} & (\S,\Pg).}
\]

In \cite{dh}, Douady and Hubbard, following Thurston, give a
complete characterization of equivalence classes of rational maps
among those of Thurston maps.  The characterization takes the
following form.

A branched covering $f: (\S, \Pf)  \to (\S, \Pf)$ induces a
holomorphic self-map
\[\sigma_f: {\rm Teich}(\S,\Pf) \to {\rm
Teich}(\S,\Pf)\] of Teichm\"uller space (see Section
\ref{prelimsect} for the definition). Since it is obtained by
lifting complex structures under $f$, we will refer to $\sigma_f$ as
the {\em pullback map} induced by $f$. The map $f$ is equivalent to
a rational map if and only if the pullback  map $\sigma_f$ has a
fixed point. By a generalization of the Schwarz lemma, $\sigma_f$
does not increase Teichm\"uller distances.  For most maps
$f$, the pullback map $\sigma_f$ is a contraction, and so a fixed
point, if it exists, is unique.

In this note, we give examples showing that the contracting behavior
of $\sigma_f$ near this fixed point can be rather varied.

\begin{theo}\label{alt_thm1}
There exist Thurston maps $f$ for which $\sigma_f$ is contracting,
has a fixed point $\tau$ and:
\begin{enumerate}
\item the derivative of $\sigma_f$ is invertible at $\tau$, the image of $\sigma_f$ is open and dense
in $\pteich$ and $\sigma_f:\pteich\to \sigma_f\bigl(\pteich\bigr)$
is a covering map,
\item the derivative of $\sigma_f$ is not invertible at $\tau$, the image of $\sigma_f$ is equal to $\pteich$
and $\sigma_f:\pteich\to \pteich$ is a ramified Galois covering map,\footnote{A ramified covering is Galois if the group of deck transformations acts transitively on the fibers. }

or
\item the map $\sigma_f$ is constant.
\end{enumerate}
\end{theo}

%That is, the fixed point $\tau$ can be respectively attracting,
%superattracting, or the only point in the image of $\sigma_f$.

In Section \ref{prelimsect}, we establish notation, define
Teichm\"uller space and the pullback map $\sigma_f$ precisely, and
develop some preliminary results used in our subsequent analysis. In
Sections \ref{pf1}, \ref{pf2}, and \ref{Xexamples}, respectively, we
give concrete examples which together provide the proof of Theorem
\ref{alt_thm1}.   We supplement these examples with some partial
general results.  In Section \ref{pf1}, we  state a fairly general
sufficient condition on $f$ under which $\sigma_f$ evenly covers it
image.   This condition, which can sometimes be checked in practice,
is excerpted from \cite{k} and \cite{k2}. Our example
illustrating (2) is highly symmetric and atypical; we are not aware
of any reasonable generalization.   In Section \ref{sigmaconst}, we
state three  conditions on $f$ equivalent to the condition that
$\sigma_f$ is constant.   Unfortunately, each is  extremely
difficult to verify in concrete examples.

{\bf{ Acknowledgements.}}  We would like to thank Curt
  McMullen who provided the example showing (3).
%The example showing (1) was
%found by Bartholdi and Nekrashevych and appeared in \cite{bn}.  The
%example showing (2) was found by Buff and Epstein, while that
%showing (3) was found by McMullen. }

\section{Preliminaries}\label{prelimsect}

Recall that a Riemann surface is a connected oriented topological
surface together with a {\em complex structure}:  a
maximal atlas of charts $\phi:U\to\C$ with holomorphic overlap maps.
For a given oriented, compact topological surface $X$, we denote the
set of all complex structures on $X$ by ${\cal{C}}(X)$. It is easily
verified that an orientation-preserving branched covering map $f:X\to
Y$ induces a map $f^*: {\cal{C}}(Y)\to {\cal{C}}(X)$; in particular,
for any orientation-preserving homeomorphism $\psi:X\to X$, there is an induced
map $\psi^*: {\cal{C}}(X)\to {\cal{C}}(X)$.

Let $A\subset X$ be finite. The Teichm\"uller space of $(X,A)$ is
\[{\rm Teich}(X,A)\eqdef {\cal{C}}(X)/{\sim_A}\]
where $c_1\sim_A c_2$ if and only if $c_1=\psi^*(c_2)$ for some
orientation-preserving homeomorphism $\psi:X\to X$ which is isotopic
to the identity relative to $A$. In view of the homotopy-lifting property, if
\begin{itemize}
\item $B\subset Y$ is finite and contains the critical value set
$V_f$ of $f$, and
\item $A\subseteq f^{-1}(B)$,
\end{itemize}
then $f^*:{\cal{C}}(Y)\to{\cal{C}}(X)$ descends to a well-defined
map $\sigma_f$ between the corresponding Teichm\"uller spaces:
\[
\xymatrix{
& {\cal{C}}(Y)\ar[d]\ar[rr]^{f^*} & & {\cal{C}}(X) \ar[d] \\
& \text{Teich}(Y,B)\ar[rr]^{\sigma_f} & & \text{Teich}(X,A) .}\]
 This map is known as the {\em pullback map} induced by $f$.

In addition if $f:X\to Y$ and $g:Y\to Z$ are orientation-preserving
branched covering maps and if $A\subset X$, $B\subset Y$ and $C\subset
Z$ are such that
\begin{itemize}
\item $B$ contains $V_f$ and  $C$ contains $V_g$,
\item $A\subseteq f^{-1}(B)$ and $B\subseteq g^{-1}(C)$,
\end{itemize}
then $C$ contains the critical values of $g\circ f$ and $A\subseteq
(g\circ f)^{-1}(C)$. Thus
\[\sigma_{g\circ f}:\text{Teich}(Z,C)\to
\text{Teich}(X,A)\] can be decomposed as $\sigma_{g\circ f} =
\sigma_f\circ \sigma_g$:
\[\xymatrix{
& \text{Teich}(Z,C)\ar[rr]^{\sigma_g} \ar@/_2pc/[rrrr]_{\sigma_{g\circ
f}} & & \text{Teich}(Y,B)\ar[rr]^{\sigma_f} & & \text{Teich}(X,A) .}
\]

For the special case of
${\text{Teich}}(\P^1,A)$, we may use the Uniformization Theorem to
obtain the following description. Given a finite set $A\subset \P^1$ we may
regard ${\text{Teich}}(\P^1,A)$ as the quotient of the
space of all orientation-preserving homeomorphisms  $\phi :
\P^1\rightarrow \P^1$ by the equivalence relation $\sim$ whereby
$\phi_1\sim\phi_2$  if there exists a M\"obius transformation $\mu$
such that $\mu\circ\phi_1=\phi_2$ on $A$, and $\mu\circ\phi_1$ is
isotopic to $\phi_2$ relative to $A$. Note that there is a natural
basepoint $\basepoint$ given by the class of the identity map on $\P^1$.
It is
well-known that ${\text{Teich}}(\P^1,A)$ has a natural topology and
complex manifold structure (see \cite{h1}).

The {\em moduli space} is the space of all injections $\psi:
A\hookrightarrow \P^1$ modulo postcomposition with M\"obius
transformations. The moduli space will be denoted as $\mod$.
If $\phi$ represents an element of ${\text{Teich}}(\P^1,A)$, the restriction
$[\phi]\mapsto \phi |_A$ induces a universal covering $\pi:
{\text{Teich}}(\P^1,A)\to{\text{Mod}}(\P^1,A)$ which is a local
biholomorphism with respect to the complex structures on
${\text{Teich}}(\P^1,A)$ and ${\text{Mod}}(\P^1,A)$.

% If $|A|>3$, and $\Theta\subseteq A$ where $|\Theta|=3$, then the
% moduli space is naturally an open subset of $(\P^1)^{A-\Theta}$.

%Let $f:\P^1\to \P^1$ be a ramified covering map, and let $A\subset
%\P^1$ be finite with $|A|\geq 3$. Let $\Theta\subseteq A$ such that
%$|\Theta|=3$; note that the moduli space is naturally an open subset
%of $(\P^1)^{A-\Theta}$. We can equivalently describe $\sigma_f$ in
%terms of the following commutative diagram. Let
%$\phi_1:(\P^1,A)\to\bigl(\P^1,\phi_1 (A)\bigr)$ be a homeomorphism
%such that $\phi|_{\Theta}=\id|_{\Theta}$. By the uniformization
%theorem, there exists a unique $\phi_2:(\P^1,A)\to
%\bigl(\P^1,\phi_2(A)\bigl)$ where $\phi_2|_{\Theta}=\id|_{\Theta}$,
%such that the following diagram commutes, and $f_{\phi_1}$ is
%holomorphic.
%\[\xymatrix{
%&(\P^1,A)\ar[d]_{f}\ar[rr]^{\phi_2} &&
%\bigl(\P^1,\phi_2(A)\bigr) \ar[d]^{f_{\phi_1}:=\phi_1\circ f\circ{\phi_2}^{-1}} \\
%&(\P^1,A)\ar[rr]^{\phi_1} && \bigl(\P^1,\phi_1(A)\bigr) }
%\]
%If $\tau_1\in\pteich$ is the equivalence class of $\phi_1$ and
%$\tau_2\in \pteich$ is the equivalence class of $\phi_2$, then
%\[\sigma_f(\tau_1)=(\tau_2).\]
%In particular, if $f:\P^1\to \P^1$ is rational, then
%$\sigma_f:\pteich\to \pteich$  fixes the basepoint $\basepoint$.

Let $f:\P^1\to \P^1$ be a Thurston map with $|\Pf|\geq 3$. For any
$\Theta\subseteq \Pf$ with $|\Theta|=3$, there is an obvious identification of
$\pmod$ with an open subset of $(\P^1)^{\Pf-\Theta}$. Assume
$\tau\in \pteich$ and let $\phi:\P^1\to \P^1$ be a homeomorphism
representing $\tau$ with $\phi|_{\Theta}=\id|_{\Theta}$. By the
Uniformization Theorem, there exist
\begin{itemize}
\item a unique
homeomorphism $\psi:\P^1\to \P^1$ representing $\tau'\eqdef
\sigma_f(\tau)$ with $\psi|_{\Theta}=\id|_{\Theta}$ and
\item  a unique rational map $F:\P^1\to \P^1$,
\end{itemize}
such that the following diagram commutes:
\[\xymatrix{
&(\P^1,\Pf)\ar[d]_{f}\ar[rr]^{\psi} &&
\bigl(\P^1,\psi(\Pf)\bigr) \ar[d]^{F} \\
&(\P^1,\Pf)\ar[rr]^{\phi} && \bigl(\P^1,\phi(\Pf)\bigr). }
\]
Conversely, if we have such a commutative diagram with $F$
holomorphic, then
\[\sigma_f(\tau)=\tau'\]
where $\tau\in \pteich$ and $\tau'\in \pteich$ are the
equivalence classes of $\phi$ and $\psi$ respectively. In particular, if
$f:\P^1\to \P^1$ is rational, then $\sigma_f:\pteich\to \pteich$
fixes the basepoint $\basepoint$.

\section{Proof of (1)}\label{pf1}

In this section, we prove that there are Thurston maps $f:\S\to \S$
such that $\sigma_f$
\begin{itemize}
\item is contracting,

\item has a fixed point $\tau\in \Teich(\S,\Pf)$ and

\item is a covering map over its image which is open and dense in $\Teich(\S,\Pf)$.
\end{itemize}
In fact, we show that this is the case when $\S=\P^1$ and $f:\P^1\to
\P^1$ is a polynomial whose critical points are all periodic. The
following is adapted from \cite{k2}.

\begin{prop}\label{prop_periodicpoly}
If $f:\P^1\to \P^1$ is a polynomial of degree $d\geq 2$ whose
critical points are all periodic, then
\begin{itemize}
\item $\sigma_f\bigl(\Teich(\P^1,\Pf)\bigr)$ is open and dense in
$\Teich(\P^1,\Pf)$ and
\item $\sigma_f:\Teich(\P^1,\Pf)\to \sigma_f\bigl(\Teich(\P^1,\Pf)\bigr)$
is a covering map.
\end{itemize}
In particular, the derivative $D\sigma_f$ is invertible at the fixed
point $\basepoint$.
\end{prop}

This section is devoted to the proof of this proposition.

Let $n=|\Pf|-3$. We will identify $\tmod$ with an open subset of $\P^n$ as follows.
First enumerate the finite postcritical points as $p_0,\ldots,p_{n+1}$.
Any point of $\tmod$ has a
representative $\psi:\Pf \hookrightarrow \P^1$ such that
\[\psi(\infty) = \infty \quad\text{and}\quad  \psi(p_0) = 0.\]
Two such representatives are equal up to multiplication by a nonzero
complex number. We identify the point in $\tmod$ with the point
\[[x_1:\ldots :x_{n+1}]\in \P^n \quad\text{where}\quad
x_1\eqdef \psi(p_1)\in \C,\ldots, x_{n+1}\eqdef \psi(p_{n+1})\in
\C.\] In this way, the moduli space $\tmod$ is identified with
$\P^{n}-\Delta$, where $\Delta$ is the {\em forbidden locus}:
\[\Delta\eqdef \bigl\{[x_1:\ldots:x_{n+1}]\in \P^n~;~(\exists
i,~x_i=0)\text{ or }(\exists i\neq j,~x_i=x_j)\bigr\}.\] The
universal cover $\pi:\pteich\to \tmod$ is then identified with a
universal cover $\pi:\pteich \to \P^n-\Delta$.

Generalizing a result of Bartholdi and Nekrashevych \cite{bn}, the thesis \cite{k} showed that when $f:\P^1\to \P^1$ is a
unicritical polynomial there is an analytic endomorphism
$g_f:\P^n\to \P^n$ for which the following diagram commutes:
\[
\xymatrix{  & \pteich\ar[d]_{\pi}\ar[rr]^{\sigma_f} & & \pteich \ar[d]^{\pi}      \\
& \P^n & & \P^n.\ar @{->}[ll]_{g_f} }
\]
We show that the same result holds when $f:\P^1\to \P^1$ is a
polynomial whose critical points are all periodic.

\begin{prop}\label{prop_periodicpoly2}
Let $f:\P^1\rightarrow \P^1$ be a polynomial of degree $d\geq 2$
whose critical points are all periodic. Set $n\eqdef |P_f|-3$. Then,
\begin{enumerate}
\item there is an analytic
endomorphism $g_f:\P^n\to \P^n$ for which the following diagram
commutes:
\[
\xymatrix{  & \pteich\ar[d]_{\pi}\ar[rr]^{\sigma_f} & & \pteich \ar[d]^{\pi} \\
& \P^n & & \P^n\ar @{->}[ll]_{g_f} }
\]
\item $\sigma_f$
takes its values in $\pteich-\pi^{-1}({\cal L})$ with ${\cal
L}\eqdef g_f^{-1}(\Delta)$,
\item $g_f(\Delta)\subseteq \Delta$ and
\item the critical point locus and the critical value locus of $g_f$ are contained in $\Delta$.
\end{enumerate}
\end{prop}

\medskip
\noindent{\em Proof of Proposition  \ref{prop_periodicpoly} assuming
Proposition  \ref{prop_periodicpoly2}:} Note that ${\cal L}$ is a codimension 1 analytic subset of $\P^n$, whence $\pi^{-1}({\cal L})$ is a codimension 1
analytic subset of $\pteich$. Thus, the complementary open sets are dense and connected.
Since $g_f:\P^n-{\cal L}\rightarrow \P^n-\Delta$ is a covering map, the compostion
\[g_f\circ \pi :\pteich-\pi^{-1}({\cal L}) \to \P^n-\Delta\]
is a covering map. Moreover,
\[\pi(\basepoint) = g_f\circ \pi \circ \sigma_f(\basepoint) = g_f\circ
\pi(\basepoint).\] By universality of the covering map
$\pi:\pteich\to \P^n-\Delta$, there is a unique map
$\sigma:\pteich\to \pteich - \pi^{-1}({\cal L})$ such that
\begin{itemize}
\item $\sigma(\basepoint) = \basepoint$ and
\item the following diagram commutes:
\[\xymatrix{  & \pteich\ar[d]_{\pi}\ar[rrr]^{\sigma} & & &\pteich-\pi^{-1}({\cal L}) \ar[dlll]^{g_f\circ \pi}      \\
& \P^n -\Delta. }
\]
\end{itemize}
Furthermore, $\sigma:\pteich\to \pteich - \pi^{-1}({\cal L})$ is a
covering map. Finally, by uniqueness we have $\sigma_f= \sigma$. \qed
\medskip

\medskip
\noindent{\em Proof of Proposition  \ref{prop_periodicpoly2}:}\\ (1)
We first show the existence of the endomorphism $g_f:\P^n\to \P^n$. We start with the definition of $g_f$.

The restriction of $f$ to $\Pf$ is a permutation which fixes
$\infty$. Denote by $\perm:[0,n+1]\to [0,n+1]$ the permutation
defined by:
\[p_{\perm(k)} = f(p_k)\]
and denote by $\perminv$ the inverse of $\perm$.

For $k\in [0,n+1]$, let
$m_k$ be the multiplicity of $p_k$ as a critical point of $f$ (if
$p_k$ is not a critical point of $f$, then $m_k\eqdef 0$).

Set $a_0\eqdef 0$ and let $Q\in \C[a_1,\ldots,a_{n+1},z]$ be the homogeneous polynomial of
degree $d$ defined by
\[Q(a_1,\ldots,a_{n+1},z)\eqdef \int_{a_{\perminv(0)}}^z \left( d\prod_{k=0}^{n+1} (w-a_k)^{m_k} \right){\rm
d}w.\] Given ${\bf a}\in
\C^{n+1}$, let $F_{\bf a}\in \C[z]$ be the monic polynomial defined
by
\[F_{\bf a}(z) \eqdef Q(a_1,\ldots, a_{n+1},z).\]
Note that $F_{\bf a}$ is the unique monic polynomial of degree $d$ which vanishes at $a_{\perminv(0)}$ and whose critical points are exactly those points $a_k$ for which $m_k>0$, counted with multiplicity $m_k$.

Let $G_f:\C^{n+1}\to \C^{n+1}$ be the
homogeneous map of degree $d$ defined by
\[G_f\left(\begin{array}{c}a_1\\\vdots\\a_{n+1}\end{array}\right) \eqdef \left(\begin{array}{c}
F_{\bf a}(a_{\perminv(1)})\\\vdots\\
F_{\bf
a}(a_{\perminv(n+1)})\end{array}\right)=\left(\begin{array}{c}
Q(a_1,\ldots,a_{n+1},a_{\perminv(1)})\\\vdots\\
Q(a_1,\ldots,a_{n+1},a_{\perminv(n+1)})\end{array}\right).\]

We claim that $G_f^{-1}\bigl({\bf 0}\bigr) = \{{\bf
0}\}$ and thus, $G_f:\C^{n+1}\to \C^{n+1}$ induces an endomorphism $g_f:\P^n\to \P^n$.
Indeed, let us consider a point ${\bf a}\in
\C^{n+1}$. By definition of $G_f$, if $G_f({\bf a}) = {\bf 0}$, then
the monic polynomial $F_{\bf a}$ vanishes at $a_0,a_1,\ldots,a_{n+1}$.
The critical points of $F_{\bf a}$ are those
points $a_k$ for which $m_k>0$. They are all mapped to $0$ and thus,
$F_{\bf a}$ has only one critical value in $\C$. All the preimages
of this critical value must coincide and since $a_0=0$, they all
coincide at $0$. Thus, for all $k\in [0,n+1]$, $a_k=0$.

Let us now prove that for all $\tau\in \pteich$, we have
\[\pi (\tau)= g_f\circ \pi\circ \sigma_f(\tau).\]
Let $\tau$ be a point in $\pteich$ and set $\tau'\eqdef
\sigma_f(\tau)$.

We will show that there is a representative $\phi$ of $\tau$ and a representative $\psi$ of $\tau'$ such that $\phi(\infty)=\psi(\infty)=\infty$, $\phi(p_0)=\psi(p_0)=0$ and
\begin{equation}\label{eq_Gfproperty}
G_f\bigl(\psi(p_1),\ldots,\psi(p_{n+1})\bigr) = \bigl(\phi(p_1),\ldots,\phi(p_{n+1})\bigr).\end{equation}
It then follows that
\[g_f\bigl([\psi(p_1):\ldots:\psi(p_{n+1})]\bigr) = [\phi(p_1):\ldots:\phi(p_{n+1})]\]
which concludes the proof since
\[\pi(\tau') = [\psi(p_1):\ldots:\psi(p_{n+1})]\quad \text{and}\quad
\pi(\tau) = [\phi(p_1):\ldots:\phi(p_{n+1})].\]

To show the existence of $\phi$ and $\psi$, we may proceed as follows.
Let $\phi$ be any representative of $\tau$ such that
$\phi(\infty) = \infty$ and $\phi(p_0) = 0$. Then, there is a
representative $\psi:\P^1\to \P^1$ of $\tau'$ and a rational map
$F:\P^1\to \P^1$ such that the following diagram commutes:
\[
\xymatrix{  & \P^1 \ar[d]_{f}\ar[r]^{\psi} & \P^1 \ar[d]^{F} \\
& \P^1 \ar @{->}[r]^\phi & \P^1. }
\]
We may normalize $\psi$ so that $\psi(\infty) = \infty$ and
$\psi(p_0)=0$. Then, $F$ is a polynomial of degree $d$. Multiplying
$\psi$ by a nonzero complex number, we may assume that $F$ is a
monic polynomial.

We now check that these homeomorphisms $\phi$ and $\psi$ satisfy the required Property
(\ref{eq_Gfproperty}).
For $k\in [0,n+1]$, set
\[x_k\eqdef \psi(p_k)\quad\text{and}\quad y_k\eqdef \phi(p_k).\]
We must show that
\[G_f(x_1,\ldots,x_{n+1}) = (y_1,\ldots,y_{n+1}).\]
Note that for $k\in
[0,n+1]$, we have the following commutative diagram:
\[
\xymatrix{  & p_{\perminv(k)} \ar@{|->}[d]_{f}\ar@{|->}[r]^{\psi} & x_{\perminv(k)} \ar@{|->}[d]^{F} \\
& p_k \ar@{|->}[r]^\phi & y_k.}
\]
Consequently,  $F(x_{\perminv(k)}) =y_k$. In particular $F(x_{\perminv(0)})=0.$
In addition, the critical points of $F$ are exactly those points $x_k$ for which $m_k>0$, counted with multiplicity $m_k$.
As a consequence, $F=F_{\bf x}$ and
\[G_f\left(\begin{array}{c}x_1\\\vdots\\x_{n+1}\end{array}\right)=
 \left(\begin{array}{c}F_{\bf x}(x_{\perminv(1)})\\\vdots\\F_{\bf x}(x_{\perminv(n+1)})\end{array}\right)=
  \left(\begin{array}{c}F(x_{\perminv(1)})\\\vdots\\F(x_{\perminv(n+1)})\end{array}\right)=
  \left(\begin{array}{c}y_1\\\vdots\\y_{n+1}\end{array}\right).\]

(2) To see that $\sigma_f$ takes its values in
$\pteich-\pi^{-1}({\cal L})$, we may proceed by contradiction.
Assume
\[\tau\in\Teich(\P^1,\Pf) \quad \text{and}\quad
\tau'\eqdef \sigma_f(\tau)\in \pi^{-1}({\cal L}).\] Then, since $\pi
= g_f\circ \pi\circ \sigma_f$, we obtain
\[\pi(\tau)= g_f\circ \pi(\tau')\in\Delta.\]
But if $\tau\in\Teich(\P^1,\Pf)$, then $\pi(\tau)$ cannot be in
$\Delta$, and we have a contradiction.

(3) To see that $g_f(\Delta)\subseteq \Delta$, assume \[{\bf
a}\eqdef (a_1,\ldots, a_{n+1})\in \C^{n+1}\] and set $a_0\eqdef 0$.
Set
\[(b_0,b_1,\ldots,b_{n+1})\eqdef \bigl(0,F_{\bf a}(a_{\perminv(1)}),\ldots,F_{\bf a}(a_{\perminv(n+1)})\bigr).\]
Then,
\[G_f(a_1,\ldots,a_{n+1}) = (b_1,\ldots, b_{n+1}).\]
Note that
\[a_i=a_j\quad\Longrightarrow\quad
b_{\perm(i)} = b_{\perm(j)}.\] In addition $[a_1:\ldots: a_{n+1}]$
belongs to $\Delta$ precisely when there are integers $i\neq j$ in
$[0,n+1]$ such that $a_i=a_j$. As a consequence,
\[[a_1:\ldots :a_{n+1}]\in \Delta\quad \Longrightarrow \quad
[b_1:\ldots:b_{n+1}]\in \Delta.\] This proves that
$g_f(\Delta)\subseteq \Delta$.

(4) To see that the critical point locus of $g_f$ is contained in
$\Delta$, we must show that ${\rm Jac}~G_f:\C^{n+1}\to \C$ does not
vanish outside $\Delta$. Since $g_f(\Delta)\subseteq \Delta$, we
then automatically obtain that the critical value locus of $g_f$ is
contained in $\Delta$.

Note that ${\rm Jac}~G_f(a_1,\ldots,a_{n+1})$  is a homogeneous
polynomial of degree $(n+1)\cdot (d-1)$ in the variables
$a_1,\ldots,a_{n+1}$. Consider the polynomial $J\in
\C[a_1,\ldots,a_{n+1}]$ defined by
\[J(a_1,\ldots,a_{n+1})\eqdef \prod_{0\leq i<j\leq n+1} (a_i-a_j)^{m_i+m_j}
\quad\text{with}\quad a_0\eqdef 0.\]
%We wish to prove that ${\rm
%Jac}~G_f$ and $J$ are equal up to multiplication by a non zero
%complex number.

\begin{lemme}
The Jacobian ${\rm Jac}~G_f$ is divisible by $J$.
\end{lemme}

\begin{proof}
 Set $a_0\eqdef 0$ and $G_0\eqdef 0$. For $j\in
[1,n+1]$, let $G_j$ be the $j$-th coordinate of
$G_f(a_1,\ldots,a_{n+1})$, i.e.
\[G_j\eqdef d\int_{a_{\perminv(0)}}^{a_{\perminv(j)}} \prod_{k=0}^{n+1} (w-a_k)^{m_k} {\rm
d}w.\] For $0\leq i<j\leq n+1$, note that setting
$w=a_i+t(a_j-a_i)$, we have
\begin{align*}
G_{\perm(j)} - G_{\perm(i)} &= d\int_{a_i}^{a_j}
\prod_{k=0}^{n+1} (w-a_k)^{m_k} {\rm d}w \\
&= d\int_{0}^{1} \prod_{k=0}^{n+1} (a_i+t(a_j-a_i)-a_k)^{m_k}\cdot
(a_j-a_i) {\rm d} t\\
&= (a_j-a_i)^{m_i+m_j+1} \cdot H_{i,j} \end{align*} with
\[H_{i,j}\eqdef d\int_0^1
t^{m_i}(t-1)^{m_j} \prod_{k\in [0,n+1]\atop k\neq i,j}
\bigl(a_i-a_k+t(a_j-a_i)\bigr)^{m_k} {\rm d}t.\] In particular,
$G_{\perm(j)} - G_{\perm(i)}$ is divisible by
$(a_j-a_i)^{m_i+m_j+1}$.

For $k\in [0,n+1]$, let $L_k$ be the row defined as:
\[L_k\eqdef \left[\frac{\partial G_k}{\partial a_1}\quad \ldots\quad
\frac{\partial G_k}{\partial a_{n+1}}\right].\] Note that $L_0$ is
the zero row, and for $k\in [1,n+1]$, $L_k$ is the $k$-th row of the
Jacobian matrix of $G_f$. According to the previous computations,
the entries of $L_{\perm(j)}-L_{\perm(i)}$ are the partial
derivatives of $(a_j-a_i)^{m_i+m_j+1}\cdot H_{i,j}$. It follows that
$L_{\perm(j)}-L_{\perm(i)}$ is divisible by $(a_j-a_i)^{m_i+m_j}$.
Indeed, $L_{\perm(j)}-L_{\perm(i)}$ is either the difference of two
rows of the Jacobian matrix of $G_f$, or such a row up to sign, when
$\perm(i)=0$ or $\perm(j)=0$. As a consequence, $ {\rm Jac}~G_f$ is
divisible by $J$.
\end{proof}

Since $\sum m_j=d-1$, an easy computation shows that the degree of
$J$ is $(n+1)\cdot (d-1)$.
%\begin{align*}
%\sum_{0\leq i<j<\leq n+1} (m_i+m_j) & =
%\sum_{j=0}^{n+1}\sum_{i=0}^{j-1} m_j + \sum_{i=0}^{n+1}\sum_{j=i+1}^{n+1} m_i\\
%&= \sum_{j=0}^{n+1}j m_j + \sum_{i=0}^{n+1} (n+1-i) m_i\\
%&= \sum_{k=0}^{n+1}k m_k + \sum_{k=0}^{n+1} (n+1-k) m_k\\
%&= (n+1)\sum_{k=0}^{n+1} m_k = (n+1)\cdot(d-1),
%\end{align*}
Since $J$ and ${\rm Jac}~G_f$ are homogeneous polynomials of the
same degree and since $J$ divides ${\rm Jac}~G_f$, they are equal up
to multiplication by a nonzero complex number. This shows that ${\rm
Jac}~G_f$ vanishes exactly when $J$ vanishes, i.e. on a subset of
$\Delta$.

This completes the proof of Proposition  \ref{prop_periodicpoly2}.
\qed
\medskip

\section{Proof of (2)}\label{pf2}

In this section we present an example of a Thurston map $f$ such
that the pullback map $\sigma_f:\pteich\to\pteich$ is a ramified Galois covering
and has a fixed critical point.

Let $f:\P^1\to \P^1$ be
the rational map defined by:
\[
f(z)=\frac{3z^2}{2z^3+1}.
\]
Note that $f$ has critical points at
$\Omega_f=\{0,1,\omega,\bar{\omega}\}$, where
\[\omega\eqdef
-1/2+i\sqrt3/2\quad\text{and}\quad \bar{\omega}\eqdef -1/2-i\sqrt3/2\] are cube roots
of unity. Notice that
\[f(0)=0,~f(1)=1,~f(\omega)=\bar{\omega}\text{ and }
f(\bar{\omega})=\omega.\] So, $\Pf=\{0,1,\omega,\bar{\omega}\}$ and
$f$ is a Thurston map. We illustrate the critical dynamics of $f$
with the following {\it{ramification portrait}}:
\[
\xymatrix{0\ar@(ur,dr)^2}\qquad\xymatrix{1\ar@(ur,dr)^2}\qquad
\xymatrix{\omega\ar@/^1.1pc/[r]^2 & \bar{\omega}\ar@/^1.1pc/[l]^2}
\]

%%%%%%%%%%%% XB %%%%%%%%%%%
\begin{figure}[htbp]
\centerline{\scalebox{.25}{\includegraphics{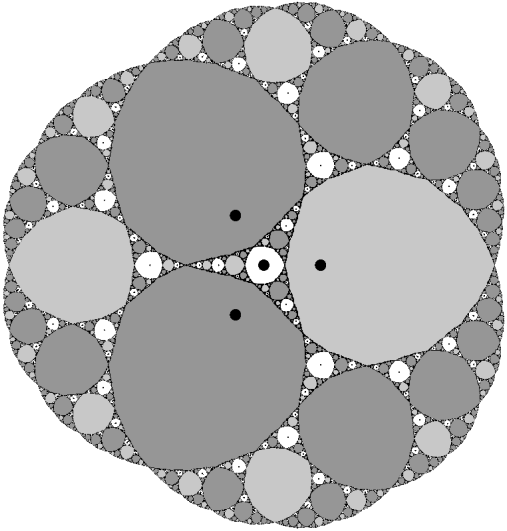}}}
\caption{The Julia set of the rational map $f:z\mapsto
3z^2/(2z^3+1)$. The basin of $0$ is white. The basin of $1$ is light
grey. The basin of $\{\omega,\bar\omega\}$ is dark grey.}
\end{figure}
%%%%%%%%%%%%%%%%%%%%%%%%%

Since $|\Pf|=4$, the Teichm\"uller space $\pteich$ has complex
dimension $1$.

Set $\Theta\eqdef \{1,\omega,\bar{\omega}\}\subset \Pf$. We identify
the moduli space $\pmod$ with $\P^1-\Theta$. More precisely, if
$\phi:\Pf\hookrightarrow \P^1$ represents a point in $\pmod$ with
$\phi|_\Theta=\id|_\Theta$, we identify the class of $\phi$ in
$\pmod$ with the point $\phi(0)$ in $\P^1-\Theta$. The universal
covering $\pi:\pteich\to \pmod$ is identified with a universal
covering $\pi:\pteich\to \P^1-\Theta$ and $\pi(\basepoint)$ is
identified with $0$.

Assume $\tau\in \pteich$ and let $\phi:\P^1\to \P^1$ be a
homeomorphism representing $\tau$ with
$\phi|_{\Theta}=\id|_{\Theta}$. There exists a unique homeomorphism
$\psi:\P^1\to \P^1$ representing $\tau'\eqdef \sigma_f(\tau)$ and a
unique cubic rational map $F:\P^1\to\P^1$ such that
\begin{itemize}
\item $\psi|_{\Theta}=\id|_{\Theta}$ and
\item the following diagram commutes
\[
\xymatrix{  & \P^1 \ar[d]_{f}\ar[r]^{\psi} &
\P^1 \ar[d]^{F} \\
& \P^1 \ar[r]^{\phi} & \P^1.}
\]
\end{itemize}
We set
\[y\eqdef \phi(0) = \pi(\tau) \quad\text{and}\quad
x \eqdef \psi(0) =\pi(\tau').\]

The rational map $F$ has the following properties:
\begin{enumerate}
\item[(P1)]\label{cond_Fcrit}{$1$, $\omega$ and $\bar \omega$ are critical points of $F$, $F(1)=1$, $F(\omega)=\bar{\omega}$, $F(\bar{\omega})=\omega$ and}
\item[(P2)]{$x\in \P^1-\Theta$ is a critical point of $F$ and $y=F(x)\in \P^1-\Theta$ is the corresponding critical value.}
\end{enumerate}

For $\alpha=[a:b]\in \P^1$, let $F_\alpha$ be the rational map
defined by
\[F_\alpha(z)\eqdef \frac{az^3+3bz^2+2a}{2bz^3+3az+b}.\]
Note that $f=F_0$.

We first show that $F=F_\alpha$ for some $\alpha\in \P^1$. For this
purpose, we may write $F=P/Q$ with $P$ and $Q$ polynomials of degree
$\leq 3$. Note that if $\widehat F=\widehat P/\widehat Q$ is another
rational map of degree $3$ satisfying Property (P1), then
$F-\widehat F$ and $(F-\widehat F)'$ vanish at $1$, $\omega$ and
$\bar \omega$. Since
\[F-\widehat F=\frac{P\widehat Q-Q\widehat P}{Q\widehat Q}\]
and since $P\widehat Q-Q\widehat P$ has degree $\leq 6$, we see that
$P\widehat Q-Q\widehat P$ is equal to $(z^3-1)^2$ up to
multiplication by a complex number.

A computation shows that $F_0$ and $F_\infty$ satisfy Property (P1).
We may write $F_0=P_0/Q_0$ and $F_\infty=P_\infty/Q_\infty$ with
\[P_0(z)=3z^2,\quad Q_0(z)=2z^3+1,\quad P_\infty(z)=z^3+2\quad\text{and}\quad
Q_\infty(z) = 3z.\] The previous observation shows that $PQ_0-QP_0$
and $PQ_\infty-QP_\infty$ are both scalar multiples of $(z^3-1)^2$,
and thus, we can find complex numbers $a$ and $b$ such that
\[a\cdot(PQ_\infty-QP_\infty)+b\cdot(PQ_0-QP_0)= 0\]
whence
\[P\cdot(aQ_\infty+bQ_0) = Q\cdot (aP_\infty+bP_0).\]
This implies that
\[F = \frac{P}{Q} = \frac{aP_\infty+bP_0}{aQ_\infty+bQ_0}=F_\alpha\quad\text{with}\quad
\alpha=[a:b]\in \P^1.\]

We now study how $\alpha\in \P^1$ depends on $\tau\in \pteich$. The
critical points of $F_\alpha$ are $1$, $\omega$, $\bar\omega$ and
$\alpha^2$. We therefore have
\[x = \alpha^2\quad\text{and}\quad y = F_\alpha(\alpha^2) = \frac{\alpha(\alpha^3+2)}{2\alpha^3
+1} = \frac{x^2+2\alpha}{2x\alpha+1}.\] In particular,
\[\alpha = \frac{x^2-y}{2xy-2}.\]

Consider now the holomorphic maps $X:\P^1\to \P^1$, $Y:\P^1\to \P^1$
and $A:\pteich\to \P^1$ defined by
\[X(\alpha)\eqdef \alpha^2,\quad Y(\alpha) \eqdef  \frac{\alpha(\alpha^3+2)}{2\alpha^3
+1}\] and
\[A(\tau) \eqdef \frac{x^2-y}{2xy-2}
\quad\text{with}\quad y=\pi(\tau)\quad\text{and}\quad x = \pi\circ
\sigma_f(\tau).\] Observe that
\[X^{-1}\bigl(\{1,\omega,\bar\omega\}\bigr) = Y^{-1}\bigl(\{1,\omega,\bar\omega\}\bigr) =
\Theta'\eqdef \{1,\omega,\bar\omega,-1,-\omega,-\bar\omega\}.\]
%In
%addition, the critical point set of $Y$ is
%$\{1,\omega,\bar\omega\}$.
Thus, we have the following commutative
diagram,
\[\xymatrix{  & \pteich \ar[dd]_{\pi}\ar[rr]^{\sigma_f} \ar[dr]^A & &
\pteich \ar[dd]^{\pi} \\
&&\P^1-\Theta' \ar[dl]_Y \ar[dr]^X &\\
& \P^1-\Theta & & \P^1-\Theta.}
\]

In this paragraph, we show that $\sigma_f$ has local degree two at
the fixed basepoint. Since $f=F_0$, we have $A(\basepoint)=0$. In
addition, $\pi(\basepoint)= \pi\circ\sigma_f(\basepoint)= 0$. Since
$Y(\alpha)= 2\alpha + {\cal O}(\alpha^2)$, the germ $Y:(\P^1,0)\to
(\P^1,0)$ is locally invertible at $0$. Since $\pi:\pteich\to \pmod$
is a universal covering, the germ
$\pi:\bigl(\pteich,\basepoint\bigr)\to
\bigl(\pmod,\basepoint\bigr)$ is also locally invertible at $0$.
Since $X(\alpha)=\alpha^2$, the germ $X:(\P^1,0)\to (\P^1,0)$ has
degree $2$ at $0$. It follows that $\sigma_f$ has degree $2$ at
$\basepoint$ as required.

Finally, we prove that $\sigma_f$ is a surjective Galois orbifold
covering. First, note that the critical value set of $Y$ is $\Theta$
whence $Y:\P^1-\Theta'\to \P^1-\Theta$ is a covering map. Since
$\pi=Y\circ A$ and since $\pi:\pteich\to P^1-\Theta$ is a universal
covering map, we see that $A:\pteich\to \P^1-\Theta'$ is a covering map (hence a universal covering map).

Second, note that $X:\P^1-\Theta'\to \P^1-\Theta$ is a ramified Galois
covering of degree $2$, ramified above $0$ and $\infty$ with local
degree $2$. Let $M$ be the orbifold whose underlying surface is
$\P^1-\Theta$ and whose weight function takes the value $1$
everywhere except at $0$ and $\infty$ where it takes the value $2$.
Then, $X:\P^1-\Theta'\to M$ is a covering of orbifolds and $X\circ
A:\pteich \to M$ is a universal covering of orbifolds.

Third, let $T$ be the orbifold whose underlying surface is $\pteich$ and
whose weight function takes the value $1$ everywhere except at
points in $\pi^{-1}\bigl(\{0,\infty\}\bigr)$ where it takes the
value $2$. Then $\pi:T\to M$ is a covering of orbifolds. We have the
following commutative diagram:
\[\xymatrix{ & \pteich \ar[drrr]_{X\circ A}\ar[rrr]^{\sigma_f} & & &
T \ar[d]^{\pi} \\
& & & & M.}
\]
It follows that  $\sigma_f:\pteich\to T$ is a covering of orbifolds
(thus a universal covering). Equivalently, the map
$\sigma_f:\pteich\to \pteich$ is a ramified Galois covering,
ramified above points in $\pi^{-1}\bigl(\{0,\infty\}\bigr)$ with
local degree $2$.

Figure \ref{fig_orbifoldcovering} illustrates the behavior of the map $\sigma_f$.
\begin{figure}[htbp]
\centerline{\begin{picture}(330,150)(0,0) \put(0,0){
\scalebox{.3}{\includegraphics{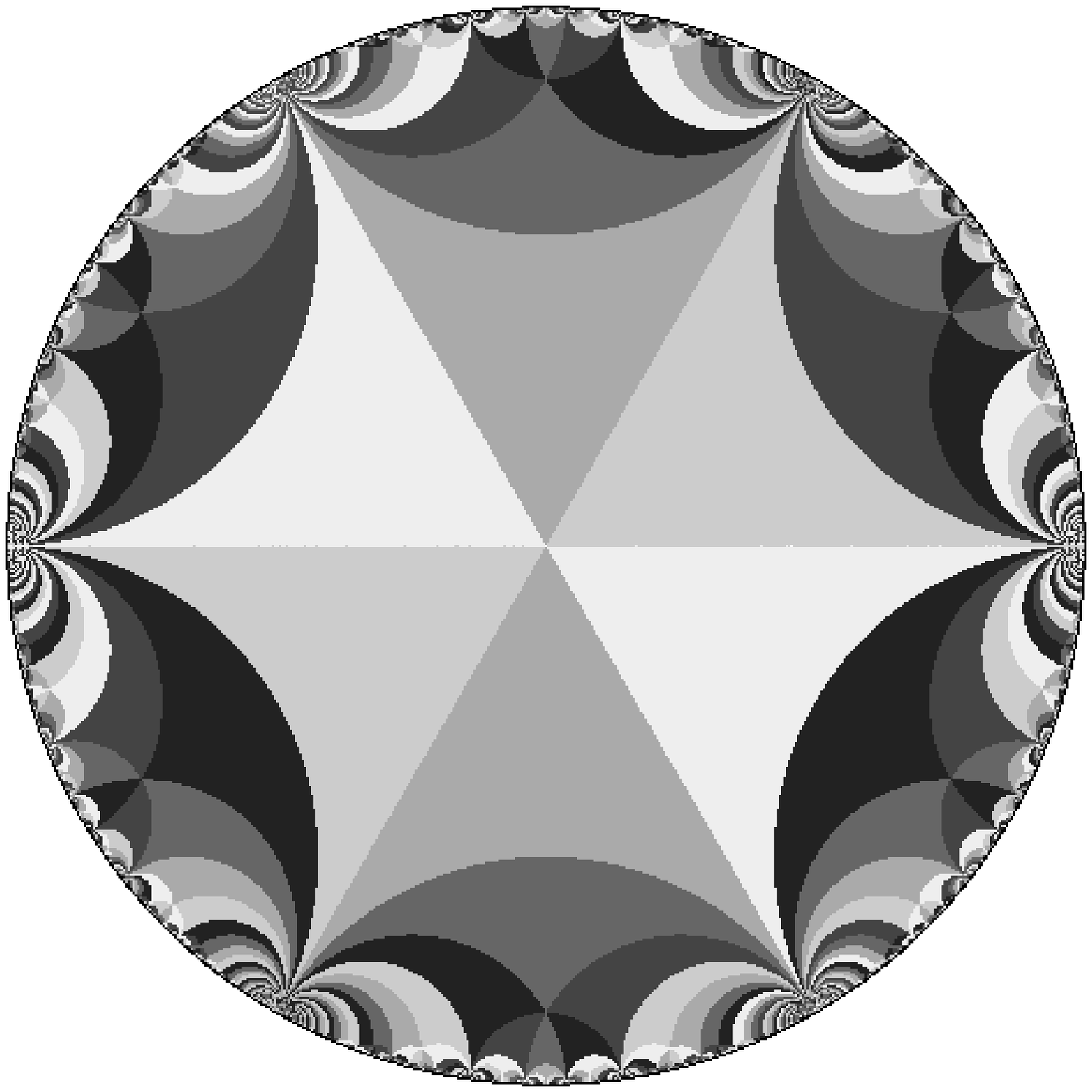}}}
\put(160,80){$\overset{\sigma_f}\longrightarrow$}
\put(180,0){\scalebox{.3}{\includegraphics{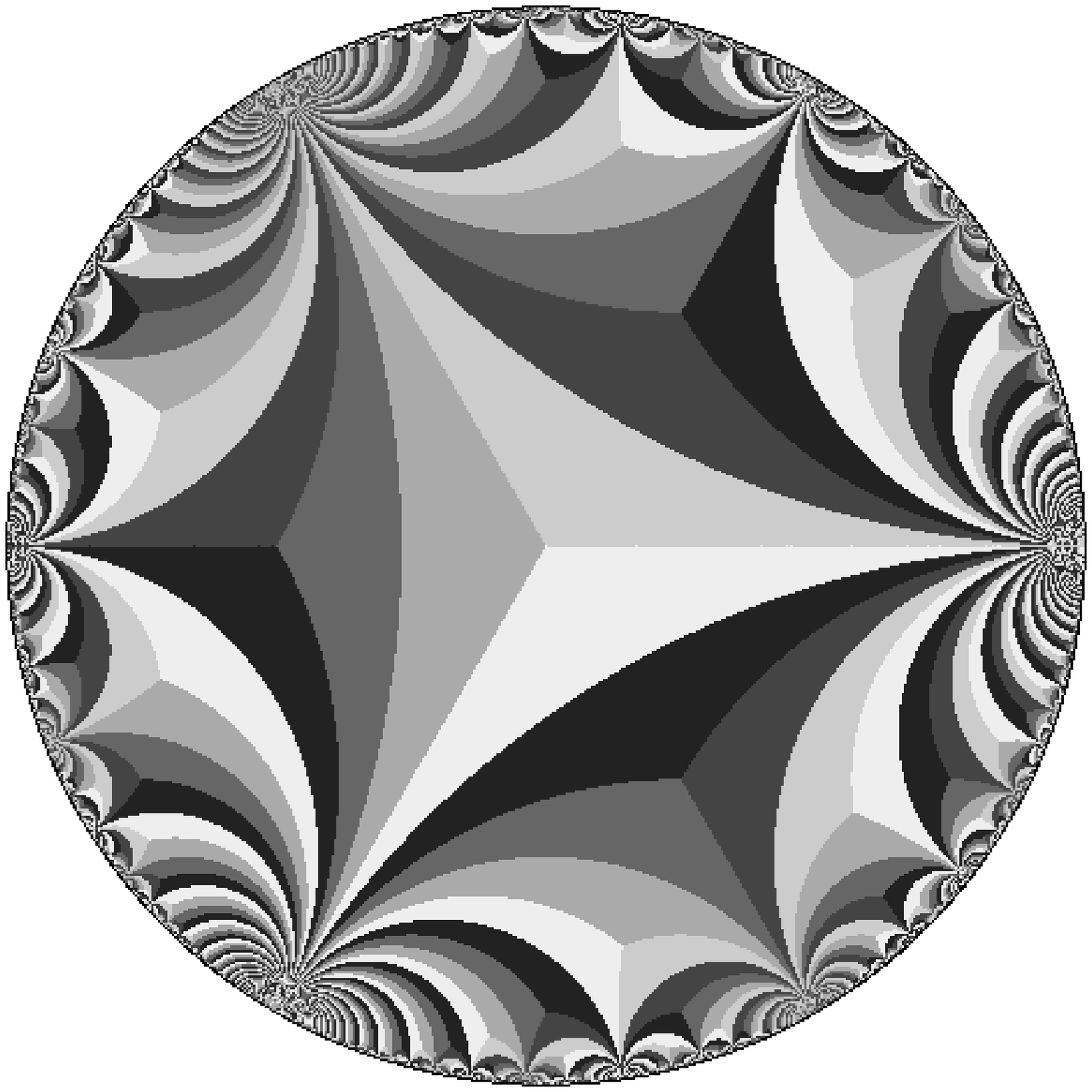}}}
\end{picture}
} \caption{For $f(z)=3z^2/(2z^3+1)$, the pullback map $\sigma_f$
fixes $0=\basepoint$. It sends hexagons to triangles. There is a
critical point with local degree $2$ at the center of each hexagon
and a corresponding critical value at the center of the image
triangle. The map $X\circ A$ sends light grey hexagons to the unit
disk in $\P^1-\Theta$ and dark grey hexagons to the complement of
the unit disk in $\P^1-\Theta$. The map $\pi$ sends  light grey
triangles to the unit disk in $\P^1-\Theta$ and dark grey triangles
to the complement of the unit disk in
$\P^1-\Theta$.\label{fig_orbifoldcovering}}
\end{figure}
\pagebreak

\section{Proof of (3)}\label{pf3}

\subsection{Examples}\label{Xexamples}

Here, we give examples of Thurston maps $f$ such that
\begin{itemize}
\item $\Pf$ contains at least $4$ points, so
$\pteich$ is not reduced to a point, and
\item $\sigma_f:\pteich\to\pteich$ is constant.
\end{itemize}
The main result, essentially due to McMullen, is the following.

\begin{prop}\label{prop_constantsigma}
Let $\belyi:\P^1\to\P^1$ and $g:\P^1\to \P^1$ be rational maps with
critical value sets $V_\belyi$ and $V_g$. Let $A\subset \P^1$ be
finite. Assume $V_\belyi\subseteq A$ and $V_g\cup g(A)\subseteq
\belyi^{-1}(A)$. Then
\begin{itemize}
 \item $f\eqdef g\circ \belyi$ is a Thurston map,
 \item  $V_g\cup g(V_\belyi)\subseteq \Pf\subseteq V_g\cup g(A)$ and
 \item the dimension of the image of
$\sigma_f:\pteich\to \pteich$ is at most $|A|-3$.
\end{itemize}
\end{prop}

\begin{rema} If $|A|=3$ the pullback map $\sigma_f$ is
constant.
\end{rema}

\begin{proof}

Set $B:=V_g\cup g(A)$. The set of critical values of $f$ is the set
\[V_f= V_g \cup g(V_\belyi)\subseteq B.\]
By assumption,
\[f(B) = g\circ \belyi(B) \subseteq g(A) \subseteq B.\]
So, the map $f$ is a Thurston map and $V_g \cup g(V_\belyi)\subseteq
\Pf\subseteq B$.

Note that $B \subseteq \belyi^{-1}(A)$ and $A\subseteq
g^{-1}(B)$. According to the discussion at the beginning of
Section \ref{prelimsect}, the rational maps $\belyi$ and $g$ induce
pullback maps
%\[\sigma_\belyi:\teich\to
%\pteich \quad \text{and}\quad \sigma_g:\pteich\to \teich.\] 
\[\sigma_\belyi:\teich\to\mathrm{Teich}(\P^1,B) \quad \text{and}\quad
\sigma_g:\mathrm{Teich}(\P^1,B)\to \teich.\]
In
addition,
\[ \sigma_f = \sigma_\belyi\circ \sigma_g.\]
The dimension of the Teichm\"uller space $\teich$ is $|A|-3$. Thus,
the rank of $D\sigma_g$, and so that of $D \sigma_f$, at any
point in $\teich$ is at most $|A|-3$. This completes the proof of
the proposition.
\end{proof}

Let us now illustrate this proposition with some examples.

\begin{exa} We are not aware of any rational map $f:\P^1\to \P^1$ of degree $2$ or
$3$ for which $|\Pf|\geq 4$ and $\sigma_f:\pteich\to\pteich$ is
constant. We have an example in degree $4$: the polynomial $f$
defined by
\[f(z) = 2i \left(z^2-\frac{1+i}2\right)^2.\]
This polynomial can be decomposed as $f=g\circ \belyi$ with
\[\belyi(z) = z^2\quad\text{and}\quad g(z) = 2i\left(z-\frac{1+i}2\right)^2.\]
See Figure 5.  The critical value set of $\belyi$ is \[V_\belyi=\{0,\infty\}\subset
A\eqdef \{0,1,\infty\}.\] The critical value set of $g$ is
\[V_g = \{0,\infty\}\subset \{0,\infty,-1,1\}=\belyi^{-1}(A).\]
In addition, $g(0)=-1$, $g(1) = 1$ and $g(\infty) = \infty$, so
\[g(A) = \{-1,1,\infty\}\subset \belyi^{-1}(A).\]
According to the previous proposition, $f=g\circ \belyi$ is a
Thurston map and since $|A|=3$, the map $\sigma_f:\pteich\to
\pteich$ is constant.

Note that $V_f=\{0,-1,\infty\}$ and $\Pf=\{0,1,-1,\infty\}$. The
ramification portrait for $f$ is:
\[ \xymatrix{ & \sqrt{\frac{1+i}{2}}\ar[dr]^2
\\ &                      & 0 \ar[r]^{2}            &-1 \ar[r]
&       1\ar@(ur,dr)[]   &           &           &\infty
\ar@(ur,dr)^4 []    \\ &-\sqrt{\frac{1+i}{2}}\ar[ur]_2.} \]
\begin{figure}[htbp]
\centerline{\scalebox{.25}{\includegraphics{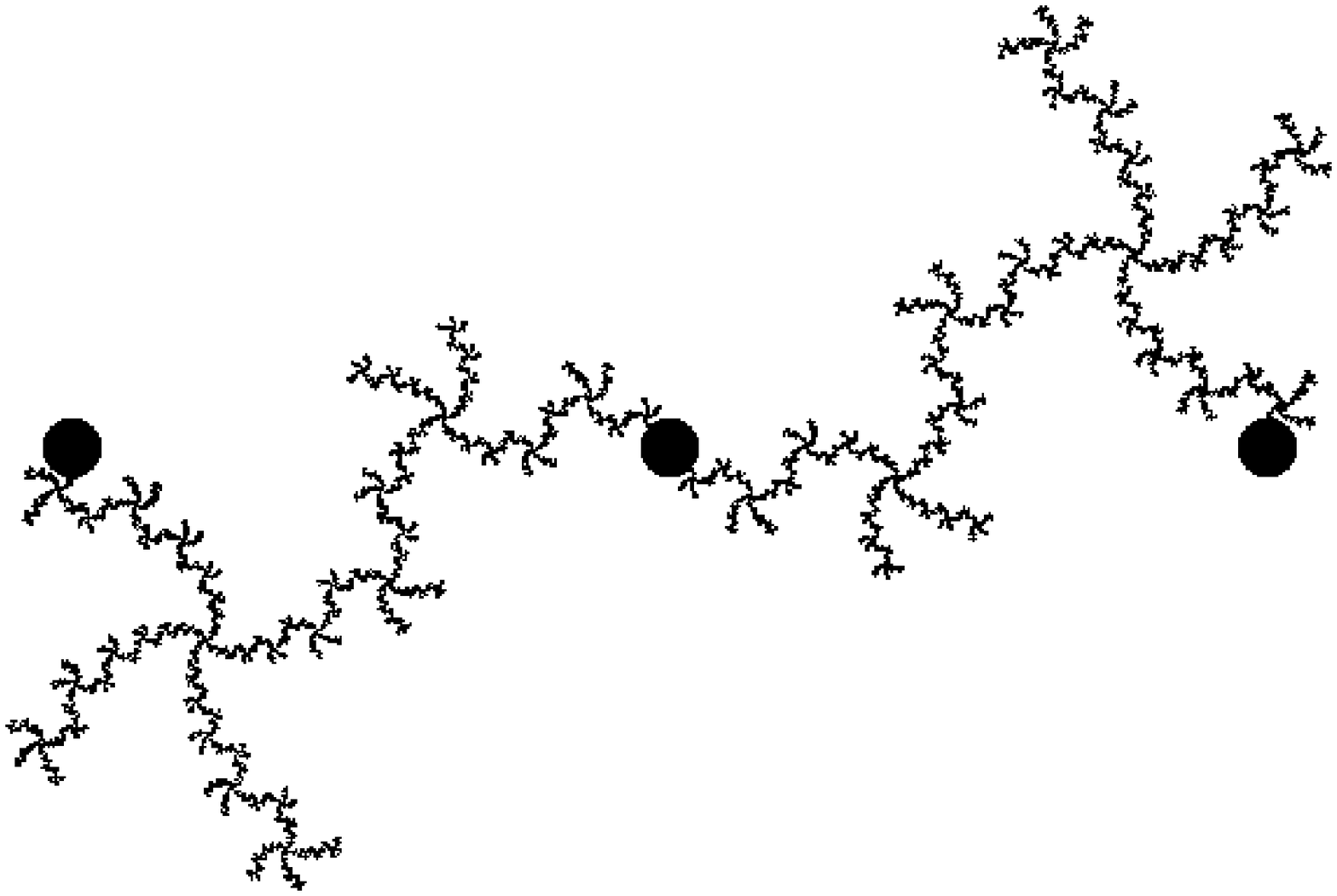}}}
\caption{The Julia set of the degree 4 polynomial $f:z\mapsto 2i
\left(z^2-\frac{1+i}2\right)^2$ is a dendrite. There is a fixed
critical point at $\infty$. Its basin is white. The point $z=1$ is a
repelling fixed point. All critical points are in the backward orbit
of $1$.}
\end{figure}
\end{exa}

\begin{exa} We also have examples of rational maps $f:\P^1\to \P^1$ for which
$\sigma_f:\pteich\to \pteich$ is constant and $|\Pf|\geq 4$ is an
arbitrary integer. Assume $n\geq 2$ and consider $\belyi:\P^1\to
\P^1$ and $g:\P^1\to \P^1$ the polynomials defined by
\[\belyi(z) = z^n\quad\text{and}\quad g(z) = \frac{(n+1)z-z^{n+1}}{n}.\]
Set $A:=\{0,1,\infty\}$. The critical value set of $\belyi$ is
$V_\belyi=\{0,\infty\}\subset A$.

The critical points of $g$ are the $n$-th roots of unity and $g$
fixes those points; the critical values of $g$ are the $n$-th roots
of unity. In addition, $g(0)=0$. Thus
\[V_g\cup g(V_\belyi) = V_g\cup g(A) = \belyi^{-1}(A).\]
According to Proposition \ref{prop_constantsigma}, $\Pf=
\belyi^{-1}(A)$ and the pullback map $\sigma_f$ is constant. In
particular, $|\Pf|=n+2$.

For $n=2$,  $f$ has the following {ramification portrait}:
\[ \xymatrix{ & i\ar[dr]^2
\\ &                      & -1 \ar[r]^2
&      1\ar@(ur,dr)[]^2   &             &0 \ar@(ur,dr)^2 & &\infty
\ar@(ur,dr)^6 []    \\ &-i\ar[ur]_2} \]
\begin{figure}[htbp]
\centerline{\scalebox{.25}{\includegraphics{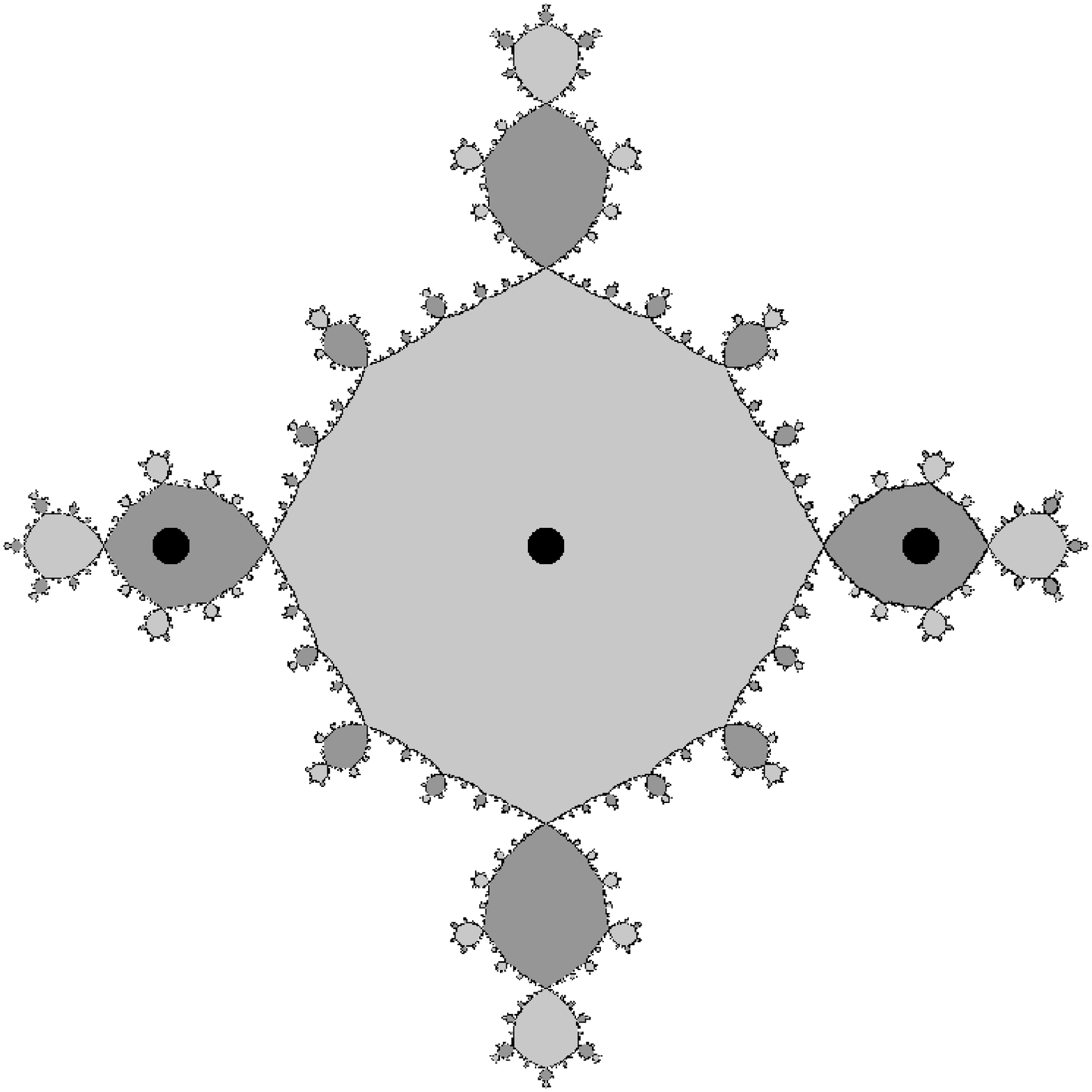}}}
\caption{The Julia set of the degree 6 polynomial $f:z\mapsto
z^2(3-z^4)/2$. There are superattracting fixed points at $z=0$,
$z=1$
  and $z=\infty$. All
  other critical points are in the backward orbit of $1$. The basin of
  $\infty$ is white. The basin of $0$ is light grey. The basin of $1$ is
  dark grey.}
\end{figure}
\end{exa}

\begin{exa}
Proposition \ref{prop_constantsigma} can be further exploited to
produce examples of Thurston maps $f$ where $\sigma_f$ has a {\em
skinny image}, which is not just a point.

For $n\geq 2$, let $A_n$ be the union of $\{0,\infty\}$ and the set
of $n$-th roots of unity. Let $\belyi_n:\P^1\to \P^1$ and
$g_n:\P^1\to \P^1$ be the polynomials defined by
\[\belyi_n(z) = z^n\quad\text{and}\quad g_n(z) = \frac{(n+1)z-z^{n+1}}{n}.\]
The critical points of $g_n$ are the $n$-th roots of unity and $g_n$
fixes those points; the critical values of $g_n$ are the $n$-th
roots of unity. In particular, $V_{g_n}\subset A_n$. In addition,
$g_n(0)=0$, and so,
\[g_n(A_n) = A_n.\]

Assume $n\geq 2$ and $m\geq 1$ are integers with $m$ dividing $n$,
let's say $n=km$. Note that
\[V_{\belyi_k}\subset A_m\quad\text{and}\quad
V_{g_n}\cup g_n(A_n) = A_n = \belyi_k^{-1}(A_m).\] It follows that
the polynomial $f:\P^1\to \P^1$ defined by
\[f:=g_n\circ \belyi_k\]
is a Thurston map and
\[A_n=V_{g_n}\cup g_n(V_{\belyi_k})\subseteq \Pf \subseteq V_{g_n}\cup
g_n(A_n)=A_n\quad\text{so},\quad P_f=A_n.\] In particular, the
dimension of the Teichm\"uller space $\pteich$ is $n-1$.

%\begin{claim} The dimension of the image of
%$\sigma_f:\pteich\to \pteich$ is $m-1$. Thus, its codimension is
%$(k-1)m$.
%\end{claim}

\noindent{\bf Claim.} {\em The dimension of the image of
$\sigma_f:\pteich\to \pteich$ is $m-1$. Thus, its codimension is
$(k-1)m$.}

\begin{proof}
On the one hand, since $g_n$ is a polynomial whose critical points
are all fixed, Proposition \ref{prop_periodicpoly} implies that
$\sigma_{g_n} : {\rm Teich}(\P^1, A_n)\to {\rm Teich}(\P^1, A_n)$
has open image. Composing with the forgetful projection \[{\rm
Teich}(\P^1,A_n)\to {\rm Teich}(\P^1,A_m),\] we deduce that
$\sigma_{g_n} : {\rm Teich}(\P^1, A_n)\to {\rm Teich}(\P^1, A_m)$
has open image.

On the other hand, since $\belyi_k:\P^1- A_n\to \P^1- A_m$ is a covering map, it follows from general principle that
$\sigma_{\belyi_k} : {\rm Teich}(\P^1, A_m)\to
{\rm Teich}(\P^1, A_n)$ is a holomorphic embedding with everywhere injective derivative.
\end{proof}
\end{exa}

%\begin{remark}
%The ideas behind those examples can be further exploited to produce
%examples of Thurston maps $f$ where $\sigma_f$ has a ``skinny''
%image, which is not just a point. For instance, if
%$g(z)=(9z^5-5z^9)/4$ and $f(z)=g(z^2)$, then
%$\Pf=\{0,1,-1,i,-i,\infty\}$, and $\mathrm{dim}(\pteich)=3$. One can
%use the ideas presented in this section to prove that the dimension
%of  $\sigma_f\bigl(\pteich\bigr)$ is equal to $1$.
%\end{remark}

\begin{question}If $f:\P^1\to\P^1$ is a Thurston map such that
the pullback map $\sigma_f:\myteich\to\myteich$ is constant, then is
it necessarily of the form described above? In particular, is there
a Thurston map $f:\P^1\to\P^1$ with constant
$\sigma_f:\myteich\to\myteich$, such that $\text{deg}(f)$ is prime?
\end{question}

\subsection{Characterizing when $\sigma_f$ is constant.}\label{sigmaconst}

Suppose $f$ is a Thurston map with $|\Pf|\geq 4$.

%\subsubsection*{{Thurston linear transformation.}}  
Let $\mathcal{S}$
denote the set of {free homotopy classes} of simple, closed, 
unoriented curves $\gamma$ in $\S-\Pf$ such that each component of $\S-\gamma$ contains at least two points of $\Pf$ . Let
$\R[\mathcal{S}]$ denote the free $\R$-module generated by
$\mathcal{S}$.  Given $[\gamma]$ and $[\widetilde{\gamma}]$ in  $\mathcal{S}$,
define the {\em pullback relation} on $\mathcal{S}$, denoted
$\pullsbackto$, by defining $[\gamma] \pullsbackto [\widetilde{\gamma}]$
if and only if there is a component $\delta$ of $f^{-1}(\gamma)$
which, as a curve in $\S-\Pf$, is homotopic to $\widetilde{\gamma}$.

The {\em Thurston linear map}
\[ \lambda_f: \R[\mathcal{S}] \to \R[\mathcal{S}]\]
is defined  by specifying the image of basis elements $[\gamma] \in
\mathcal{S}$ as follows:
\[ \lambda_f\bigl([\gamma]\bigr) = \sum_{[\gamma] \pullsbackto [\gamma_i]} d_i [\gamma_i].\]
Here,  the sum ranges over all $[\gamma_i]$ for which $[\gamma]
\pullsbackto [\gamma_i]$, and
\[ d_i = \sum_{f^{-1}(\gamma)\supset \delta \simeq \gamma_i}\frac{1}{|\deg(\delta \to \gamma)|}, \]
where the sum ranges over components  $\delta$ of $f^{-1}(\gamma)$
homotopic to $\gamma_i$.

%\subsubsection{{Virtual endomorphism.}}  
Let $\pmcg$ denote the pure
mapping class group of $(\P^1, \Pf)$---that is, the quotient of the
group of orientation-preserving homeomorphisms fixing $\Pf$
pointwise by the subgroup of such maps isotopic to the identity
relative to $\Pf$. Thus,
\[\pmod = \pteich/ \pmcg.\]

Elementary covering space theory and homotopy-lifting facts imply
that there is a finite-index subgroup $H_f \subset \pmcg$ consisting of
those classes represented by homeomorphisms $h$ lifting under $f$ to
a homeomorphism $\tilde{h}$ which fixes $\Pf$ pointwise.  This
yields a homomorphism
\[ \phi_f: H_f \to \pmcg \]
defined by
\[\phi_f\bigl([h]\bigr)=[\tilde h]\quad \text{with}\quad h\circ f =  f \circ \tilde h. \]
{Following \cite{bn}} we
refer to the homomorphism $\phi_f$ as the {\em virtual
endomorphism of $\pmcg$} associated to $f$.

\begin{theo}\label{tfae}
The following are equivalent:
\begin{enumerate}
\item $\pullsbackto$ is empty
\item $\lambda_f$ is constant
\item $\phi_f$ is constant
\item $\sigma_f$ is constant
\end{enumerate}
\end{theo}

In \cite{BEKP}, there is a mistake in the proof that $(2)\implies (3)$.   The assumption (2) is equivalent to the assumption that every curve, when lifted under $f$, becomes inessential or peripheral.  Even if this holds, it need not be the case that every Dehn twist lifts under $f$ to a pure mapping class element.  We give an explicit example after the proof of Theorem \ref{tfae}.

\begin{proof}

In \cite{BEKP} the logic was: $(1)\implies(2)\implies(3)\implies(4)$, and  failure of $(1)$ implies failure of $(4)$. 

Here is the revised logic: $(1)\iff (2)$, $(3)\iff (4)$, $(3)\implies (2)$, and failure of $(4)$ implies failure of $(1)$. 

$\mathbf{(1)\iff (2)}$ This follows immediately from the definitions.  

$\mathbf{(3)\implies (2)}$  We show failure of $(2)$ implies failure of $(3)$. If $\lambda_f$ is not constant, then there exists a simple closed curve $\gamma$ which has an essential, nonperipheral  simple closed curve $\delta$ as a preimage under $f$.  Some power of the Dehn twist about $\gamma$ lifts under $f$ to a product of nontrivial Dehn twists.  The hypothesis implies that the lifted map is homotopically nontrivial, so $\phi_f$ is not constant.

For the remaining implications, we will make use of the following facts.  

First, recall that the deck group $\pmcg$ of $\pi: \pteich \to \pmod$ acts by pre-composition properly discontinuously and biholomorphically on the space $\pteich$.  For $h \in \pmcg$ and $\tau \in \pteich$ we denote by $h\cdot \tau$ the image of $\tau$ under the action of $h$.  Since $H_f$ has finite index in $\pmcg$, the covering map $ \pteich/H_f \to \pmod$ is finite.  Furthermore, the definitions imply 
\[ \sigma_f(h\cdot \tau) = \phi_f(h)\cdot \sigma_f(\tau)\ \forall \ h \in H_f.\]

Second, {\em a bounded holomorphic function on a finite cover of $\pmod$ is constant}.  To see this, recall that $\pmod$ is isomorphic to the complement of a finite set of hyperplanes in $\C^n$ where $n=|P_f|-3$.  Let $L$ be any complex line not contained in the forbidden locus.  The intersection of $L$ with $\pmod$ is isomorphic to a compact Riemann surface punctured at finitely many points.  If $\widetilde{L}$ is any component of the preimage of $L$ under the finite covering, then $\widetilde{L}$ is also isomorphic to a compact Riemann surface punctured at finitely many points.   By Liouiville's theorem, the function is constant on $\widetilde{L}$. Since $L$ is arbitrary, the function is locally constant, hence constant.  

$\mathbf{(3)\implies(4)}$
Suppose (3) holds.  Then  $\sigma_f: \pteich \to \pteich$ descends
to a holomorphic map 
\[
\overline{\sigma}_f: \pteich/H_f \to \pteich.
\]

A theorem of Bers \cite[Section 6.1.4]{IT} shows that $\pteich$ is isomorphic to a bounded domain of $\C^n$, so $\sigma_f$ is constant.  

$\mathbf{(4)\implies (3)}$  Suppose $h \in H_f$.   If $\sigma_f \equiv \tau$ is constant, the deck transformation defined by $\phi_f(h)$ fixes the point $\tau$, hence must be the identity.  So $\phi_f$ is constant.  

$\mathbf{\mbox{\bf not} (4) \implies \mbox{\bf not}(1)}$  We first prove a Lemma of perhaps independent interest.

\begin{lemme} Let $f:(S^2,P)\to (S^2,P)$ be a Thurston map. Then the image of $\sigma_f$ is either a point, or unbounded in ${\cal{M}}_P$. 
\end{lemme}

\begin{proof} 
The definitions imply that $\pi \circ \sigma_f$ descends to a holomorphic map 
\[ \rho:\pteich/H_f\to \pmod \hookrightarrow \C^n.\] 
If the image is bounded, the map $\rho$ is constant.  
\end{proof}

\bigskip 

Suppose now that $\sigma_f$ is not constant (ie, failure of $(4)$). The Lemma implies that the image of $\pi \circ \sigma_f$ is unbounded; in particular, ${\cal{M}}_P':=\pi(\sigma_f(\pteich))$ is not contained in any compact subset of $\pmod$. This means that there exists a point $x\in {\cal{M}}_P'$ corresponding to a Riemann surface $X:=\P^1-Q$ containing an annulus $A$ of large modulus. Because $x\in {\cal{M}}_P'$, there exists a  a rational map 
\[
F: (\P^1,Q) \to (\P^1, R).
\]
such that the diagram in the definition of $\sigma_f$ commutes.  
Let $X':=X-F^{-1}(R)$ and $Y=\P^1-R$, so that $F:X'\to Y$ is a holomorphic covering map. Let $A':=A\cap X'$. There is an embedded subannulus $B'\subseteq A'$ of large modulus because we removed at most $d\cdot |P_f|$ points from $A$ to get $A'$. Hence in the hyperbolic metric on $X'$, the core curve of $B'$ is very short. Call this core curve $\delta$. Look at $F(\delta)$. Since $F:X'\to Y$ is a local hyperbolic isometry, the length of $F(\delta)$ is at most $d$ times length of $\delta$, so $F(\delta)$ is also very short. Let $\gamma$ be the geodesic in the homotopy class of $F(\delta)$.  Since $\gamma$ is very short, it must be simple.  Since $\delta$ is essential and non-peripheral, so is $\gamma$.  We conclude that $\gamma \leftarrow_f \delta$, hence $\leftarrow_f$ is nonempty. 
\end{proof}

Let $f=g\circ s$ be the quartic polynomial in Example 1.
Let $\gamma_0$ be the boundary of a small regular neighborhood $D$ of the segment $[0,1] \subset \C$.  Let $h_0: \P^1 \to \P^1$ be the right Dehn twist about $\gamma_0$.

\noindent{\bf Claim.}  {\em If $h_1: \P^1 \to \P^1$ satisfies $h_0 \circ f = f \circ h_1$ (i.e. $h_1$ is a lift of $h_0$ under $f$) then $h_1 \not\in \pmcg$. See Figure 5.} 
%MANUAL REFERENCE

\begin{figure}[htbp]
\label{fig:ex1}
\centerline{\scalebox{0.75}{\includegraphics{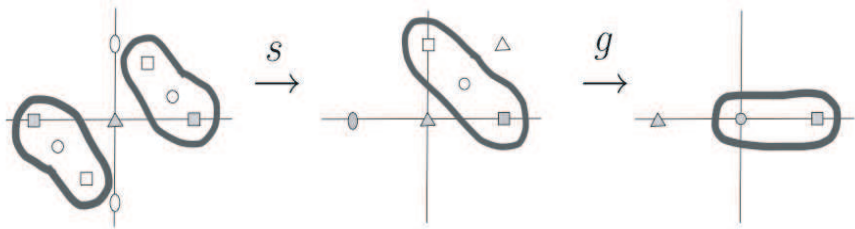}}}
\caption{The mapping properties of $f=g\circ s$ in Example 1.  The points in grey are $-1, 0, +1$.}
\end{figure}

\begin{proof} We argue by contradiction.

We may assume $h_0$ is supported on an annulus $A_0$ surrounding a bounded Jordan domain $D_0$ whose boundary is $\gamma_0$, and an unbounded region $U_0$.  Easy calculations show that the inverse image of $D_0$ under $f$ consists of two bounded Jordan domains $D_1^\pm$ each mapping as a quadratic branched cover onto $D_0$ and ramified at the points $c_\pm:=\pm \sqrt{\frac{1+i}{2}}$ (the positive sign corresponding to the root with positive real part), both of which map to the origin under $f$.  The domain $D_1^+$ contains two preimages of the point $1$, namely $+1$ and $+\frac{1+i}{\sqrt{2}}$, while its twin $D_1^-$ also contains two preimages of the point $1$, namely $-1$ and $-\frac{1+i}{\sqrt{2}}$.  The points $\pm 1 \in D_1^\pm$ belong to $P_f$, so if $h_1 \in \pmcg$ is a lift of $h_0$, then $h_1(1)=1$ and $h_1(-1)=-1$.   

Since $f: D_1^\pm - \{c_\pm\} \to D_0 - \{0\}$ are both unramified coverings, and $h_0: D_0-\{0\} \to D_0-\{0\}$ is the identity map, we conclude $h_1: D_1^\pm - \{c_\pm\} \to D_1^\pm - \{c_\pm\}$ is a deck transformation of this covering fixing a point, hence is the identity on $D_1^\pm$.

The preimage of the annulus $A_0$ is a pair of disjoint, non-nested  annuli $A_1^\pm$ with an inner boundary component $\gamma_1^\pm$ equal to $\bdry D_1^\pm$.  Since $f: A_1^\pm \to A_0$ is quadratic and unramified, and, by the previous paragraph, the restriction $h_1|_{D_1^\pm}= \id_{\gamma_1^\pm}$, we must have $h_1 \neq \id$ on the outer boundary components of $A_1^\pm$; indeed, $h_1$ there effects a half-twist.

The preimage of $U_0$ under $f$ is a single unbounded region $U_1$, which is homeomorphic to the plane minus two disks and three points; it maps in a four-to-one fashion, ramified only at the origin.  The restriction $f: U_1 - \{f^{-1}(0)\} \to U_0-\{-1\}$ is an unramified covering map, so $h_1: U_1 - \{f^{-1}(-1)\} \to U_1 - \{f^{-1}(-1)\}$ is a deck transformation of this covering.  By the previous paragraph, it is distinct from the identity.

We will obtain a contradiction by proving that $h_1: U_1 - \{f^{-1}(-1)\} \to U_1 - \{f^{-1}(-1)\}$ has a fixed point; this is impossible for deck transformations other than the identity.  We use the Lefschetz fixed point formula.  By removing a neighborhood of $\infty$ and of $-1$, and lifting these neighborhoods, we place ourselves in the setting of compact planar surfaces with boundary, so that this theorem will apply.  Under $h_1$, the boundary component near infinity is sent to itself, as are the outer boundaries of $A_1^\pm$ and the boundary component surrounding the origin (since the origin is the uniquely ramified point of $f$ over $U_0$).  The remaining pair of boundary components are permuted amongst themselves.  The action of $h_1: U_1 - \{f^{-1}(-1)\} \to U_1 - \{f^{-1}(-1)\}$ on rational homology has trace equal to either $3$ or $5$.  A fixed point thus exists, and the proof is complete.  
\end{proof}

{\bf Remark:}  There exists a lift $h_1$ of $h_0$ under $f$.  First, there is a lift $h'$ of $h_0$ under $g$, obtained by setting $h' = \id$ on the preimage of $U_0$.  This extends to a half-twist on the preimage $A_0'$  of $A_0$ under $g$, which then in turn extends to a homeomorphism fixing the preimage $D_0'$ of $D_0$ under $g$; inside $D_0'$, this homeomorphism interchanges the points $1, i$ which are the primages of $1$.  It is then straightforward to show that $h'$ lifts under $s$ by setting $h_1=\id$ on $U_1$ and extending similarly over the annuli $A_1^\pm$ and the domains $D_1^\pm$.

\pagebreak

%---END AUTHOR EDIT---

%*******************************************************************************
% Back Matter
%   Back matter begins with the bibliography.  We strongly recommend using a
%   .bib file.  You are welcome to use another bibliography style instead of our
%   in-house style.  For final production, the generated file can be renamed to
%   bibliography.tex. This file can be edited as needed, usually to fix breaks.
%
%   The index follows.  The index is generated using \makeindex (see above).
%   Use \printindex to display the index.
%
%   For final production, the generated file can be renamed to index.tex.
%   This file can be edited as needed.  Make sure that the \makeindex and
%   \printindex commands remain commented-out for the final version of the
%   index.
%*******************************************************************************
%\backmatter
% \renewcommand{\chaptermark}[1]{\markboth{#1}{#1}}
%\bibliographystyle{amsalpha}
%\bibliography{BuffEpstein}

\begin{thebibliography}{99}

\bibitem[BN] {bn} {\sc  L. Bartholdi} $\&$ {\sc V. Nekrashevych},
{\em Thurston equivalence of topological polynomials}, Acta. Math.
197: (2006) 1-51.

\bibitem[BEKP]{BEKP}{\sc X. Buff, A. Epstein, S. Koch, and K. Pilgrim}, 
{\em On Thurston's pullback map}, in Dierk Schleicher, ed., {\em Complex Dynamics: Family and Friends}, 561-583.  Wellesley, MA: AK Peters Ltd., 2009.  

%\bibitem[DD] {dd} {\sc A. Douady} $\&$ {\sc R. Douady},
%{\em Alg\`ebre et th\'eories galoisiennes}. 2. CEDI, Paris, 1979.

\bibitem[DH] {dh} {\sc  A. Douady} $\&$ {\sc J.H.  Hubbard},
{\em A proof of Thurston's characterization of rational functions}.
Acta Math. 171(2): (1993) 263-297.

%\bibitem[E] {eps} {\sc A. Epstein},
%{\em Towers of Finite Type Complex Analytic Maps}, PhD thesis, CUNY
%(1993).

\bibitem[H] {h1} {\sc J. H. Hubbard},
{\em Teichm\"uller Theory and applications to geometry, topology,
and dynamics, volume 1: Teichm\"uller theory}, Matrix Editions,
2006.

\bibitem[IT]{IT}{\sc Y. Imayoshi, M. Taniguchi}, {\em An introduction to Teichm\"uller spaces}.  New York: Springer, 1992.  

%\bibitem[H2] {h2} {\sc J. H. Hubbard},
%{\em Teichm\"uller Theory and applications to geometry, topology,
%and dynamics, volume 2}, Matrix Editions, to appear.

\bibitem[K1] {k} {\sc  S. Koch},
{\em Teichm\"uller theory and endomorphisms of $\P^n$}, PhD thesis,
Universit\'e de Provence (2007).

\bibitem[K2] {k2} {\sc  S. Koch},
PhD thesis, Cornell University, in preparation.

%\bibitem[M] {mc} {\sc C. T. McMullen},
%{\em Riemann surfaces, dynamics and geometry; Math 275 course
%notes}, Harvard University, (1998).

%\bibitem[P] {p1} {\sc K. Pilgrim},
%{\em An algebraic formulation of Thurston's combinatorial
%equivalence}, Proceedings of the AMS 131(11): (2003) 3527-3534.
%\bibitem[P2] {p2} {\sc K. Pilgrim} {\em Canonical Thurston obstructions}, Adv. Math. 158: (2001) 154-168.

\end{thebibliography}

%\include{bibliography} % use and edit for final version
%\small
%
%\printindex % uncomment to display the index during authoring
%\include{index} % use and edit for final version

% any unnumbered pages go here, typically a color insert.
% \pagestyle{plain}
% \fancyhf{}
\end{document}